\documentclass[times,doublespace]{nlaauth}
\usepackage{moreverb}
\usepackage{subcaption}
\newcommand\BibTeX{{\rmfamily B\kern-.05em \textsc{i\kern-.025em b}\kern-.08em
T\kern-.1667em\lower.7ex\hbox{E}\kern-.125emX}}
\usepackage[english]{babel}
\usepackage{pgf}
\usepackage{tikz}
\usepackage[utf8]{inputenc}
\usepackage{amsmath,amssymb,amsfonts}
\usepackage{enumerate}
\usepackage{blkarray}
\usepackage{cite}
\usepackage{subcaption} 
\usepackage{mathtools}
\usepackage{algorithm}
\usepackage{enumerate}
\usepackage[noend]{algpseudocode}
\makeatletter
\def\BState{\State\hskip-\ALG@thistlm}
\makeatother
\usetikzlibrary{calc,fit,matrix,arrows,automata,positioning,patterns}
\usepackage{hyperref}
\usepackage{xcolor}
\hypersetup{
    colorlinks,
    linkcolor={blue!50!black},
    citecolor={blue!50!black},
    urlcolor={blue!80!black}
}
\newcommand{\B}[1]{\mbox{\boldmath $#1$}}

\usepackage{pgfplots}
\usepackage{hyperref}

\DeclareMathOperator{\diag}{\mathcal D}

\newtheorem{theorem}{Theorem}[section]

 \newtheorem{example}[theorem]{Example}
  \newtheorem{definition}[theorem]{Definition}

\usepackage{siunitx}
\begin{document}

\runningheads{A.~Aristodemo, L.~Gemignani}{Accelerating the Sinkhorn-Knopp iteration by  Arnoldi-type  methods}

\title{Accelerating the Sinkhorn-Knopp iteration by  Arnoldi-type  methods}

\author{A.~Aristodemo\affil{1}, L.~Gemignani\affil{4}\corrauth}

\address{\affilnum{1} Dipartimento di Matematica, Universit\`a di Pisa,
 Largo B. Pontecorvo, 5 - 56127 Pisa, Italy\break
 \affilnum{2} Dipartimento di Informatica, Universit\`{a} di Pisa,
Largo Bruno Pontecorvo, 3 - 56127 Pisa, Italy. E-mail: luca.gemignani@unipi.it}

\corraddr{luca.gemignani@unipi.it}

\cgs{The work of L.~Gemignani was partially supported by the GNCS/INdAM project
 ``Tecniche  Innovative per Problemi di Algebra Lineare'' and by the University
of Pisa (grant PRA 2017-05). }

\begin{abstract}
  It is shown that the problem of balancing  a nonnegative matrix by positive diagonal matrices can be
  recast as a constrained nonlinear  multiparameter eigenvalue problem.  Based on this equivalent formulation
  some adaptations of the  power method and  Arnoldi process are proposed for  computing the  dominant eigenvector which
  defines the structure of the diagonal transformations. Numerical results illustrate
  that our novel  methods  accelerate significantly  the convergence of the  customary Sinkhorn-Knopp iteration
 for matrix balancing  in the case  of clustered dominant eigenvalues. 
\end{abstract}

\keywords{Sinkhorn-Knopp iteration,  Nonlinear Eigenvalue Problem, Power method, Arnoldi method}

\maketitle

\section{Introduction}
Many important types of  data,  like text,  sound, event logs,
biological sequences,  can be viewed
as graphs  connecting basic data elements.   Networks  provide a powerful   tool  for describing the dynamic behavior of
systems in biology, computer science, information engineering.
Networks and graphs  are generally represented as  very large nonnegative  matrices   describing either the
network topology, quantifying certain attributes of nodes or exhibiting the correlation between certain node features.
Among the challenging  theoretical and computational problems  with these matrices there are  the  balancing/scalability issues. 

The Sinkhorn-Knopp (SKK)  balancing problem  can be stated as follows:  Given a nonnegative matrix
$A\in \mathbb R^{n\times n}$ $(A\geq 0)$,
find if they  exist
two nonnegative diagonal matrices $D_1, D_2\in \mathbb R^{n\times n}$ such that  $S=D_1 A D_2$ is doubly stochastic, i.e.,
\begin{equation}\label{skkmio}
 D_2 A^TD_1 \B e=\B e, \quad D_1 AD_2\B e=\B e, \quad \B e=\left[1,\ldots, 1\right]^T.
 \end{equation}
 The problem  was raised in three  different papers \cite{SK0,SK1,SK2} that 
 contain the well-known iteration for matrix balancing that bears their names.
 Several equilibration problems exist 
in which row or column norms are not equal but rather are specified by positive vectors.
Variants of the  SKK problem 
have attracted attention
in various fields of pure and applied sciences  including
input-output
analysis in economics \cite{AP},  optimal transportation  theory  and its applications
in machine learning \cite{CU},  complex network analysis \cite{KN,BF},  probabilistic and statistical modeling  \cite{RL},
optimization of traffic and telecommunication  flows \cite{LS} and matrix preconditioning \cite{DB}.
For a general review and summary of these applications  one can see  \cite{IDEL}.

For any admissible vector $\B v\in  \mathbb R^n$ and any $\alpha\in \mathbb Z$, $\alpha \neq 0$,
let $\mathcal D^{\alpha}(\B v)$ be
defined as the  $n\times n$ diagonal matrix  with diagonal entries $d_i=v_i^{\alpha}$, $1\leq i\leq n$.
Then the  computation  in \eqref{skkmio} 
amounts to find two vectors $\B r$ and $\B c$ such that $D_1=\mathcal D(\B r)$ and $D_2=\mathcal D(\B c)$ satisfy
\[
\left\{ \begin{array}{ll}
  \mathcal D(\B c)A^T\mathcal D(\B r)\B e=\mathcal D(\B c)A^T\B r=\mathcal D(A^T\B r)\B c=\B e; \\
  \mathcal D(\B r)A\mathcal D(\B c)\B e=\mathcal D(\B r)A\B c=\mathcal D(A\B c)\B r=\B e
\end{array}
\right.
\]
When $A$ is symmetric we can  determine  $\B r=\B c=\B z$  to satisfy $\mathcal D(A\B z)\B z= \mathcal D(\B z) A\B z=\B e$.
In \cite{SK1} the authors proposed the following fixed point iteration --called Sinkhorn-Knopp (SKK) iteration--
for computing  the desired vectors  $\B r$ and $\B c$:
\begin{equation}\label{e1}
\left\{ \begin{array}{ll}
 \B c_{k+1}=\mathcal D^{-1}(A^T\B r_k)\B e; \\
  \B r_{k+1}=\mathcal D^{-1}(A\B c_{k+1})\B e
\end{array}
\right.
\end{equation}
In the symmetric case  the  SKK iteration reduces to
\begin{equation}\label{e2}
\B z_{k+1}= \mathcal D^{-1}(A\B z_k)\B e
\end{equation}
or, equivalently, by setting $1./\B z_k=\B x_k$ with the assumption $1/0=+\infty$,
\begin{equation}\label{e3}
\B x_{k+1}= 1./\B z_{k+1} = A \B z_k=A (1./\B x_k).
\end{equation}
The SKK iterations \eqref{e1},\eqref{e2},\eqref{e3} have been rediscovered several times in different applicative contexts. 
Related methods are  the RAS method \cite{AP} in economics,  the iterative proportional fitting procedure (IPFP) in statistics
and Kruithof's projection scheme \cite{LS}  in optimization.

A common drawback of all  these iterative algorithms  is
the slow convergence behavior  exhibited even in deceivingly simple cases.
To explain this performance gap we observe that
the equations in \eqref{e1} can be combined to  get
\begin{equation}\label{e3_single}
  \B c_{k+1}=\mathcal D^{-1}(A^T \mathcal D^{-1}(A\B c_k)\B e)\B e, \quad k\geq 0, 
\end{equation}
which can be  expressed componentwise as
\[
(\B c_{k+1})_s=\left(\sum_{m=1}^n a_{m,s}
\left(\sum_{\ell=1}^na_{m,\ell}(\B c_k)_\ell\right)^{-1}\right)^{-1}, \quad 1\leq s\leq n, \ k\geq 0.
\]
This means that \eqref{e3_single}  is equivalent to the fixed point iteration
\begin{equation}\label{e3_single_1}
 \B c_{k+1}=T( \B c_{k}), \quad T(\B x_s)=\left(\sum_{m=1}^n a_{m,s}
 \left(\sum_{\ell=1}^na_{m,\ell}\B x_\ell\right)^{-1}\right)^{-1},
\end{equation}
for solving
\begin{equation}\label{fpp}
  \B x= T\B x, \quad \B x\geq \B 0,
  \end{equation} 
where $T$ is the nonlinear operator introduced by Menon in \cite{Menon2,Menon1} and  according
to those papers we write $T\B x=T(\B x)$.

Our  first  contribution  consists of  a  novel formulation of the  fixed point problem  \eqref{fpp}
as a constrained nonlinear multiparameter  eigenvalue problem of the form
\begin{equation}\label{nlevp}
\B x= J_T(\B x)\B x, \quad \B x\geq \B 0,
\end{equation}
where $J_T(\B z)$ denotes   the Jacobian matrix  of $T$ evaluated at the point $\B z$. Although
the proof is quite simple, to our knowledge   this property has been   completely overlooked in the literature
even though  it has several implications. 

From a theoretical viewpoint, it follows that the local dynamics of the
original SKK  algorithm  \eqref{e3_single} can be described as a  power method with perturbations \cite{Ste}
applied to the  matrix $J_T(\B x)$ evaluated at the  fixed point.  Therefore  the
SKK iterations \eqref{e1},\eqref{e2},\eqref{e3} inherit the pathologies of the power  process
in the case of clustered dominant eigenvalues of  $J_T(\B x)$. A related  result  has appeared in \cite{KN}.

Furthermore,  relation  \eqref{nlevp} can also be exploited  in order to speed up the  computation of the
Sinkhorn-Knopp vector.  Acceleration  methods  using nonlinear solvers applied to equation \eqref{fpp}  have been  recently
proposed in  \cite{KR}   whereas 
 optimization strategies and descending techniques  are considered  in \cite{KKS,PL}. 
  In this paper  we pursue a different approach  by  taking into account 
the properties of the  equivalent nonlinear  multiparameter eigenvalue problem \eqref{nlevp}.
Krylov methods are the algorithms of choice for the computation of a few eigenvalues of largest magnitude of
matrices.  They have  been efficiently used in  information retrieval and web search engines for  accelerating
PageRank computations \cite{YYN, GoGr}.  In particular,   Arnoldi-based methods have been proven to be efficient
for achieving
eigenvalue/eigenvector separation \cite{GoGr}. Based on this, 
we propose  here to  compute an approximation of the fixed point of $T$ by using a different  fixed point iteration method
of the form 
$ \lambda_k\B v_{k+1}= J_T(\B v_{k})\B v_{k+1}$, $k\geq 0$,  where $\lambda_k$ is the dominant eigenvalue of
$J_T(\B v_{k})$ with corresponding normalized eigenvector $\B v_{k+1}$. Each iteration amounts to  approximate
 the  dominant eigenpair of a matrix $J_T(\B v_{k})$. 
Fast eigensolvers relying upon  the power method and  the Arnoldi process
  are  specifically tailored to solve these problems for large-scale matrices.
 Numerical results show that the resulting schemes are successful attempts to accelerate the convergence of 
 the SKK iterations in the case of  clustered dominant eigenvalues of $J_T(\B x)$.

 The paper is organized as follows. In Section \ref{sec2}  after briefly recalling the properties of  the
 SKK fixed point iteration \eqref{e3_single_1}  we  exploit the eigenvalue connection by  devising accelerated
 variants of \eqref{e3_single_1}  using Arnoldi-type  methods.   The description and implementation of these variants
 together with numerical results are discussed in Section \ref{sec3}.  Finally,
 in section \ref{sec4} conclusion and some remarks on
 future work are given.

\section{Theoretical Setup}\label{sec2}
Let  us denote by $\mathcal P, \mathcal P_0$ and $\mathcal P_\infty$  the subsets of
$\bar {\mathbb R}^n$, $\bar {\mathbb R}=\mathbb R \cup \{ \pm \infty \}$, defined by
$\mathcal P=\{\B x\in \mathbb R^n\colon \B x \geq \B 0\}$, $\mathcal P_0=\{\B x\in \mathbb R^n\colon \B x > \B 0\}$ and
$\mathcal P_\infty=\{\B x\in \bar{\mathbb R}^n\colon \B x \geq \B 0\}$, respectively.  For the sake of simplicity
we assume that $A\in \mathbb R^{n\times n}$
is  a matrix with all  positive entries,  that is, $A>0$.  Results  for more general
indecomposable nonnegative
matrices  are obtained by a continuity argument using the  perturbative analysis  introduced in
\cite{Menon1} (see also Section 6.2 in \cite{LN_book}).  Numerical evidences  shown in Section \ref{sec3}
indicate that our approach also works in the more
general setting.

Arithmetic operations are 
generalized  over    the nonnegative extended real line  $[0, +\infty]\subset \bar {\mathbb R}$ by setting \cite{Menon1}
$1/0=\infty$, $1/\infty=0$, $\infty +\infty=\infty$, $0\cdot \infty=0$, $a\cdot \infty=\infty$ if $a>0$,
where  $\infty=+\infty$. Under these assumptions we can introduce the nonlinear operators defined as follows:
\begin{enumerate}
\item $U\colon \mathcal P_\infty \rightarrow \mathcal P_\infty$, $U\B x=1./\B x$;
\item  $S\colon \mathcal P_\infty \rightarrow \mathcal P_\infty$, $S\B x= U A \B x$;
\item  $T\colon \mathcal P_\infty \rightarrow \mathcal P_\infty$, $T\B x= U A ^T U A \B x$.
\end{enumerate}
In this way  it can be easily noticed that $T$ is the  same as the operator introduced in \eqref{e3_single_1} and, therefore,
the Sinkhorn-Knopp problem for the matrix $A$ reduces to computing  the fixed points of $T$, that is,
the vectors $\B x \in \mathcal P_\infty$  such that
\begin{equation}\label{fp}
  \B x=T \B x= U A ^T U A \B x, \quad \B x\in \mathcal P_\infty.
  \end{equation}
Summing up the results stated in \cite{Menon2,Menon1} we obtain the following theorem concerning
the existence and the uniqueness of the  desired fixed point.
\begin{theorem}
  Let $A\in \mathbb R^{n\times n}$ be  a matrix with all positive entries.
  Then $\forall \B u \in \mathcal P\setminus \{\B 0\}$ we have $\sup\{\lambda\colon T\B u\geq \lambda \B u\}\leq 1$.  Moreover, 
  $T$ has a distinct  eigenvalue equal to $1$ with a
  unique (except   for   positive scalar multiples) corresponding   eigenvector $\B x\in \mathcal P_0$.
\end{theorem}
The basic  SKK algorithm  proceeds to approximate the eigenvector $\B x\in \mathcal P_0$ by means of
the fixed point iteration
\begin{equation}\label{skk}
  \left\{\begin{array}{ll}
  \B x^{(0)}\in \mathcal P_0;\\
  \B x^{(k+1)}=T \B x^{(k)}, \quad k\geq 0\end{array}
  \right.
\end{equation}
The iteration is shown to be globally convergent  since $T$ is  a 
contraction for the Hilbert  metric associated to
the cone $\mathcal P$.
\cite{BPS,LN}.
\begin{theorem}
  For any   $\B x^{(0)}\in \mathcal P_0$  there exists $\gamma=\gamma(\B x^{(0)}) \in \mathbb R$, $\gamma>0$,
  such that
  \[
  \lim_{k\rightarrow \infty}\B x^{(k)}=\gamma \B x.
  \]
\end{theorem}
The  convergence  is linear and the rate depends on the second singular value of the  doubly stochastic matrix
$P=\mathcal D(S\B x) A \mathcal D(\B x)$. We have the following \cite{KN}.
\begin{theorem}\label{rate}
  Let $\B x\in \mathcal P_0$  denote the limit of the sequence $\{ \B x^{(k)} \}_{k\in \mathbb N}$  generated according to
  \eqref{skk}. Then the matrix $\Sigma=\mathcal D(S\B x) A \mathcal D(\B x)$ is doubly stochastic  and, moreover,
  if $\sigma_2$ is the second largest singular value of $\Sigma$ it holds
  \[
  \parallel   \B x^{(k+1)}- \B x \parallel_2 \leq \sigma_2^2\parallel   \B x^{(k)}- \B x \parallel_2 +
  o\left(\parallel   \B x^{(k)}- \B x \parallel_2\right), \quad k\geq 0.
  \]
  \end{theorem}
The convergence can  be very slow in the case of nearly decomposable matrices. The following definition is provided in
\cite{AF,Minc}. 
\begin{definition}
  For a given $\epsilon >0$, the  matrix $A\in \mathbb R^{n\times n}$ is
  $\epsilon$-nearly decomposable if  there exist $E\in [0,1]^{n\times n}$ and  a permutation matrix $P$ such that
  $PAP^T=\hat A + \epsilon E$ where $\hat A$ is block triangular with  square diagonal blocks.
\end{definition}
The relevance of nearly decomposable matrices for the study of  dynamic systems in economics 
has been examined   by Simon and Ando \cite{SA}. The role of near decomposability  in queuing and computer system applications
has been discussed in \cite{Cou}.   For a   general overview of the properties of nearly decomposable graphs and networks
with  applications in data science and information retrieval one can see \cite{nuio}.

\begin{example}\label{ex1}
  Let $A=\displaystyle\left[\begin{array}{cc}1-\epsilon  & \epsilon \\\epsilon  & 1-\epsilon
  \end{array}\right]$, $0<\epsilon<1$ be a doubly stochastic matrix. The vector $\B x=[1,1]^T$  provides  a solution of the SKK
  scaling problem. The  singular values of the matrix $A=\mathcal D(S\B x) A \mathcal D(\B x)$  satisfy
  $\sigma_1=1$ and $\sigma_2=|1-2\epsilon|$. For the matrix
  $A=\displaystyle\left[\begin{array}{cc}1 & \epsilon \\1 & 1
  \end{array}\right]$, $\epsilon>0$, the SKK iteration \eqref{skk} is convergent but  
  the number of iterations   grows  exponentially as $\epsilon$  becomes small.  In Table \ref{t1} we  show
  the number of iterations performed by Algorithm \ref{alg1} in Section \ref{sec3}
  applied to the matrix $A$ with
  the error   tolerance 
  $\tau=1.0e-8$.
  
\begin{table}[ht]
\caption{Number  of  SKK iterations $ItN$ for different values of $\epsilon =10^{-k}$}
\centering 
\begin{tabular}{|c| c| c| c|c|c|c|c|c|c|c|}
\hline
$k$ & 1 & 2& 3& 4& 5 & 6& 7& 8& 9& 10\\ 
\hline
$ItN$ & 16& 46& 132& 391& 1139& 3312& 9563& 27360& 77413& 216017\\
\hline
\end{tabular}
\label{t1}
\end{table}
\end{example}

The local dynamics of \eqref{skk} depend on the properties of the  Jacobian matrix evaluated at the
fixed point.  By using the chain rule for the composite function  we obtain that
\[
J_T(\B z)=J_{UA^TUA}(\B z)=J_U(A^TUA\B z)\cdot J_A^T(UA\B z)\cdot J_U(A\B z) \cdot J_A(\B z).
\]
Since $J_U(\B z)=-\diag^{-2}(\B z)$ we find that
\begin{equation}
  J_T(\B z)=\diag^2(T\B z)\cdot A^T \cdot \diag^2(S\B  z) A.
\end{equation}
The next result gives a lower bound for  the spectral radius of $J_T(\B z)$ for $\B z\in \mathcal P_0$.
\begin{theorem}\label{norm}
  For any   given  fixed $\B z\in \mathcal P_0$   the spectral radius of  $J_T(\B z)$ satisfies
  $\rho(J_T(\B z))\geq 1$.
\end{theorem}
\begin{proof}
   Let us denote $G=\diag(T\B z)\cdot A^T \cdot \diag(S\B  z)$.   It holds
  \begin{align*}
    J_T(\B z) &=\diag^2(T\B z)\cdot A^T \cdot \diag^2(S\B  z) A\\
    &=\diag(T\B z) \cdot G \cdot G^T \diag^{-1}(T \B z), 
  \end{align*}
  and, hence $J_T(\B z)$ and  $G \cdot G^T$ are similar.
  Now  observe that
  \[
  G \B e=\diag(T\B z)\cdot A^T \cdot \diag(S\B  z)\B e=\diag(T\B z)\cdot A^TUA\B z=\B e.
  \]
  It follows that $\parallel G\parallel_2\geq 1$.  By using  the SVD of $G$ it is found that
  $\sigma_1(G)\geq 1$ and therefore  the spectral radius of  $J_T(\B z)$ satisfies
  $\rho(J_T(\B z))\geq 1$.  
\end{proof}
If $\B x =T\B x$, $\B x\in \mathcal P_0$,  then it is worth noting that
\[
J_T(\B x)=\diag^2(T\B x)\cdot A^T \cdot \diag^2(S\B  x) A=\diag^2(\B x)\cdot A^T \cdot \diag^2(S\B  x) A,
\]
and, hence, 
\[
J_T(\B x)=\diag(\B x) \cdot P^T P \diag^{-1}(\B x),
\]
where $P$ is introduced in  Theorem \ref{rate}. 
This means that  $J_T(\B x)$ and $F= P^T P$ are  similar and therefore the   eigenvalues of $J_T(\B x)$ are the squares of the singular values of $P$.  Since $A>0$  then it is irreducible and primitive
and the same holds for $P$ and a fortiori for $F$.
By the
Perron-Frobenius theorem  it follows that  the spectral radius of  $J_T(\B x)$ satisfies
$\rho( J_T(\B x))=1$ and  $\lambda=1$ is a simple  eigenvalue of $J_T(\B x)$
with a positive corresponding eigenvector.

A characterization of such an eigenvector can be derived by the following result.
\begin{theorem}\label{equiv}
  For each vector $\B z\in \mathcal P_0$ it holds
  \[
  T\B z=J_T(\B z)\cdot \B z.
  \]
  \end{theorem}
  \begin{proof}
    Let $\B z\in \mathbb R^n$, $\B z>\B 0$,  then we have
    \begin{align*}
   J_T(\B z)\cdot \B z 
   &=\diag^2(T\B z)\cdot A^T \cdot \diag^2(S\B  z) A \B z
   \\&=\diag^2(T\B z)\cdot A^T \cdot S\B z
   \\&=\diag^{-2}(A^T S \B z)\cdot A^T \cdot S\B z
   \\&=\diag^{-1}(A^T S \B z)\B e
   \\&=\diag(UA^T U A  \B z)\B e
   \\&= T\B z.
   \end{align*}
  \end{proof}
  This theorem implies that
  \[
  \B x\in \mathcal P_0, \ \B x=T\B x \quad \iff \quad \B x\in \mathcal P_0, \ \B x=J_T(\B x)\B x
  \]
  and therefore the eigenvector of $J_T(\B x)$  corresponding with the eigenvalue $1$ is exactly the  desired  solution of the
  SKK problem.  Furthermore, the SKK iteration \eqref{skk} can equivalently be written as
  \begin{equation}\label{skkm}
  \left\{\begin{array}{ll}
  \B x^{(0)}\in \mathcal P_0;\\
  \B x^{(k+1)}= T \B x^{(k)}=J_T(\B x^{(k)}) \B x^{(k)}, \quad k\geq 0\end{array}
  \right.
\end{equation}  
In principle one can accelerate the convergence    of this iteration without improving the efficiency of the  iterative method 
by replacing $T$ with the operator  $T_\ell=T\circ T \circ \cdots \circ T$, $T_1=T$,  generated from
the composition ($\ell$ times) of  $T$ for a certain $\ell\geq 1$.
The linearized form of the resulting  iteration  around the fixed point
$\B x=T \B x$, $\B x\in \mathcal P_0$, is 
 \begin{equation}\label{skkma1}
  \left\{\begin{array}{ll}
  \B x^{(0)}\in \mathcal P_0;\\
  \B x^{(k+1)}= J_T^\ell(\B x) \B x^{(k)}, \quad k\geq 0\end{array}
  \right.
 \end{equation}
 This is the power method applied to  the matrix $J_T^\ell(\B x)$ for the approximation of an eigenvector associated with the
 dominant eigenvalue $\lambda=1$.  A normalized  variant of \eqref{skkma1}   can be  more suited for numerical computations
 \begin{equation}\label{skkma1n}
\left\{\begin{array}{ll}
  \B x^{(0)}\in \mathcal P_0;\\
  \left\{\begin{array}{ll}
      \B v^{(k+1)}=J_T^\ell(\B x) \B x^{(k)},\\
      \B x^{(k+1)}=v^{(k+1)}/(\B e^T \B v^{(k+1)})
    \end{array}\right., \quad k\geq 0
\end{array}
  \right.
 \end{equation}
 Since  $\lambda=1$ is the  simple  dominant  eigenvalue of $J_T(\B x)$
with a positive corresponding eigenvector it is well known that  \eqref{skkma1n} generates sequences such that
 \[
 \lim_{k\rightarrow \infty}\B x^{(k)}=\B x/(\B e^T\B x), \quad
 \limsup_{k\rightarrow \infty}\parallel\B x^{(k)}-\B x/(\B e^T\B x)\parallel_2^{1/k}\leq \lambda_2^\ell=\sigma_2^{2\ell}, 
 \]
 where $0\leq \lambda_2=\sigma_2^2<1$ is the second largest eigenvalue of $J_T(\B x)$ and
 $\sigma_2$ denotes  the second largest singular value of $P$ defined as in Theorem \ref{rate}.
 For practical purposes  we introduce the following modified  adaptation of \eqref{skkma1n} called SKK$_\ell$ iteration: 
\begin{equation}\label{skkma1nm}
\left\{\begin{array}{ll}
  \B x^{(0)}\in \mathcal P_0;\\
  \left\{\begin{array}{ll}
      \B v^{(k+1)}=J_T^\ell(\B x^{(k)}) \B x^{(k)},\\
      \B x^{(k+1)}=v^{(k+1)}/(\B e^T \B v^{(k+1)})
    \end{array}\right., \quad k\geq 0
\end{array}
  \right.
 \end{equation}
For $\ell=1$ SKK$_1$ reduces to the scaled  customary SKK iteration.  Under suitable assumptions we can show that
SKK$_\ell$  generates a sequence converging to the desired fixed point.
\begin{theorem}
  Let $\{ \B x^{(k)}\}_k$  be the sequence generated by SKK$_\ell$ from a given initial guess $\B x^{(0)}\in \mathcal P_0$.
  Let $\B x\in \mathcal P_0$ be such that  $\B x=T\B x$ and $\B e^T\B x=1$. 
  Assume that:
  \begin{enumerate}
  \item $\exists \eta >0$ $\colon$ $J_T^\ell(\B x^{(k)})=J_T^\ell(\B x) +E_k$,  $\parallel E_k\parallel_2\leq \eta \sigma_2^{2\ell k}$, \ $k\geq 0$;
  \item    $\exists \gamma  >0$ $\colon$  $\parallel \prod_{k=0}^mJ_T^\ell(\B x^{(k)})\parallel_2\geq \gamma$, \ $m\geq 0$.
  \end{enumerate}
  Then  we have
  \[
  \lim_{k\rightarrow \infty} \B x^{(k)}=\B x,
  \]
  and
  \[
  \limsup_{k\rightarrow \infty} \parallel  \B x^{(k)}-\B x \parallel_2^{1/k}\leq \sigma_2^{2\ell}.
  \]
\end{theorem}
\begin{proof}
  Since $\sum_{k=0}^{\infty}\parallel E_k\parallel_2<\infty$ from Theorem 4.1 in \cite{Ste} we obtain that
  the  matrix sequence  $P_m=\prod_{k=0}^mJ_T^\ell(\B x^{(k)})$  is such that
  \[
 \lim_{m\rightarrow \infty} P_m=\B x \B z^T , \quad \B z\in \mathcal P.
  \]
  From  Property 2 in view of the continuity of the norm it follows that $\B z\neq \B 0$  and this implies the  convergence of
  $\{ \B x^{(k)}\}_k$.
  About the  rate of convergence we  observe that
  \begin{align*}
  \parallel
  \B x^{(k+1)}-\B x\parallel_2&=\parallel
  \frac{P_k\B x^{(0)}}{\B e^TP_k\B x^{(0)}}-\B x \frac{\B z^T \B x^{(0)}}{\B z^T\B  x^{(0)}}  \parallel_2\\
  &\leq  \parallel
  \frac{P_k\B x^{(0)}}{\B z^T \B x^{(0)}}-\B x \frac{\B z^T \B x^{(0)}}{\B z^T \B x^{(0)}} \parallel_2 +
  \parallel
  \frac{P_k\B x^{(0)}}{\B e^T P_k\B x^{(0)}}-\frac{P_k\B x^{(0)}}{\B z^T \B x^{(0)}} \parallel_2 \\
  &\leq \frac{\parallel (P_k  -\B x \B z^T) \B x^{(0)})\parallel_2}{|{\B z^T \B x^{(0)}}|}
  + \parallel P_k\B x^{(0)}\parallel_2
  \left|\frac{\B e^T( \B x \B z^T - P_k\B) x^{(0)} }{\B e^T P_k\B x^{(0)} \B z^T \B x^{(0)}}\right|
  \end{align*}
  which says that $\B x^{(k)}$ approaches $\B x$ as fast as $P_k$ tends to $\B x \B z^T$.
  Again using Theorem 4.1 in \cite{Ste}  under our assumptions there follows that 
   \[
  \limsup_{k\rightarrow \infty} \parallel P_k-\B x \B z^T   \parallel_2^{1/k}\leq \sigma_2^{2\ell}=\lambda_2^\ell.
  \]
  which concludes the proof.
\end{proof}
This theorem shows that in our model the speed of convergence increases as $\ell$  increases.
Also, notice that  for any $\B z\in \mathcal P_0$ the  matrix $J_T(\B z)$ is primitive and  irreducible and therefore
by the Perron-Frobenius theorem its spectral radius is a dominant eigenvalue with a corresponding positive eigenvector.
From Theorem \ref{norm} this eigenvalue is greater than or equal to 1.  It follows that for large $\ell$   the iterate
$\B  x^{(k+1)}$
provides  an approximation of the  positive dominant eigenvector of $J_T(\B  x^{(k)})$. 
This  fact suggests  to consider SKK$_{\infty}$  as an effective method for approximating the limit vector $\B x$. The method
performs as an inner-outer procedure.
In the inner phase  given the current approximation  $\B x^{(k)}$ of $\B x$ we apply the   Power Method 
 \begin{equation}\label{skkma1nnn}
\left\{\begin{array}{ll}
  \B v^{(0)}=\B x^{(k)};\\
  \left\{\begin{array}{ll}
      \B z^{(k+1)}=J_T(\B x^{(k)} ) \B v^{(k)},\\
      \B v^{(k+1)}=\B z^{(k+1)}/(\B e^T \B z^{(k+1)})
  \end{array}\right., \quad k\geq 0
\end{array}
  \right.
 \end{equation} 
 until convergence  to find the new approximation  $\B x^{(k+1)}= \B v^{(\hat k+1)}$.
 Numerically this  latter vector  solves
 \begin{equation}\label{mat_eig}
 J_T(\B x^{(k)} )\B x^{(k+1)} = \theta_k\B x^{(k+1)} , \ \theta_k=\rho( J_T(\B x^{(k)} )), \  \B e^T \B x^{(k+1)}=1.
 \end{equation}
 Ideally,  $\theta_k$ would  converge from above to $1$ as well as   the corresponding  positive
 eigenvector $\B x^{(k+1)}$ would  approach  $\B x$; also, the convergence  of the outer iteration
 should  be superlinear.

\begin{example}\label{ex2}
   As in Example \ref{ex1} let 
  $A=\displaystyle\left[\begin{array}{cc}1 & \epsilon \\1 & 1
     \end{array}\right]$ with $\epsilon=1.0e-8$. In Figure 1a  e 1b we  illustrate the
   convergence history  of  iteration \eqref{mat_eig} applied to $A$ with  starting guess
   $\B x^{(0)}=\B v^{(0)}=\left[1/2, 1/2\right]^T$.  The iterative scheme stops after 13 steps. 
    The dominant eigenpair $(\lambda_k, \B v^{(k+1)})$  is computed by using the  function {\tt eig}  of MatLab. 
    In Figure 1b we  show   the distance  between two consecutive  normalized eigenvectors
    measured in the  Hilbert metric $
    d_H(\B u, \B v)= \max_{i,j}\log\left(\displaystyle\frac{u_iv_j}{v_i u_j}\right), \ \forall \B u,\B v \in \mathcal P_0.$
   \begin{figure}
  \begin{subfigure}[b]{0.4\textwidth}
    \includegraphics[width=\textwidth]{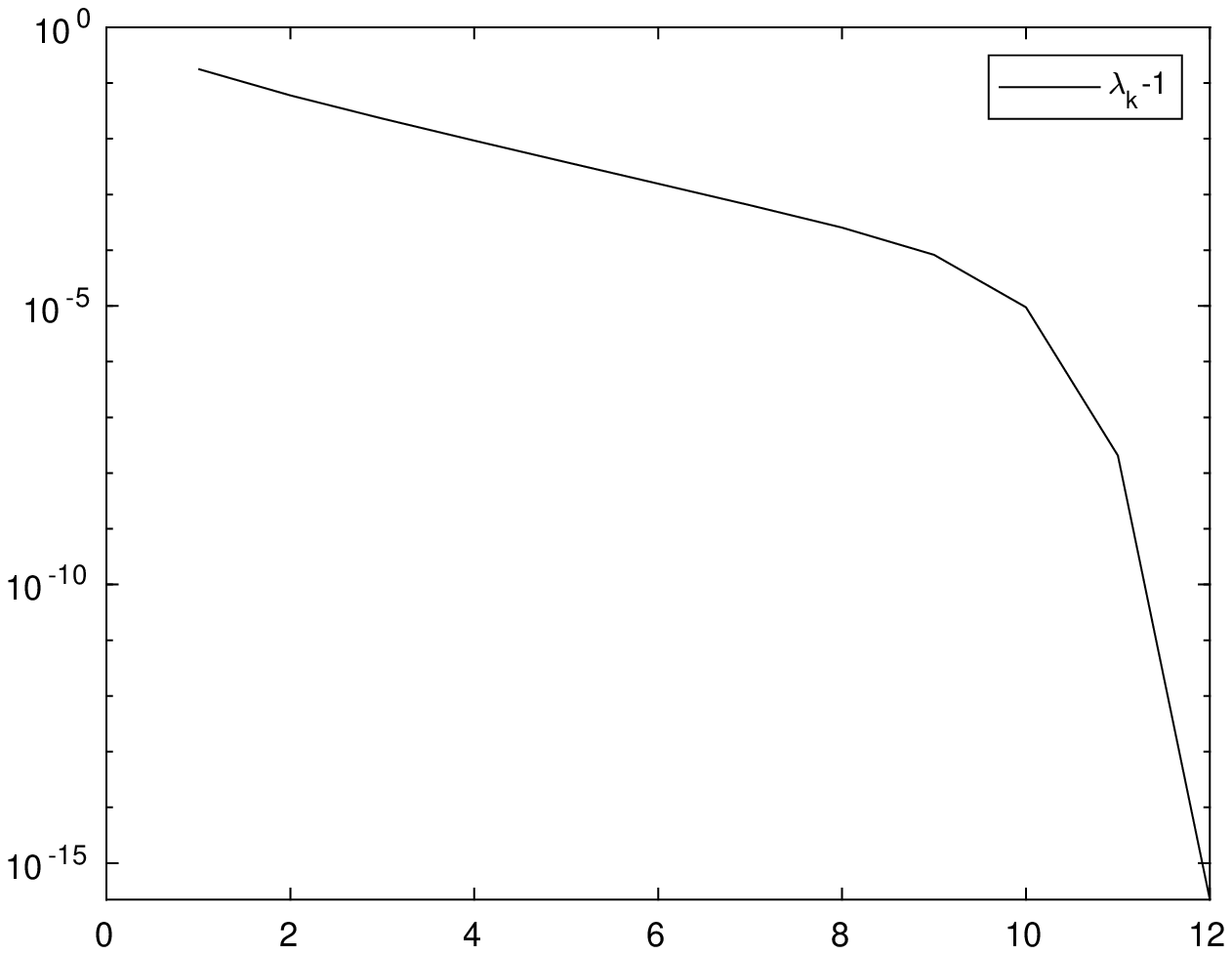}
    \caption{Figure 1a}
  \end{subfigure}
  \hfill
  \begin{subfigure}[b]{0.4\textwidth}
    \includegraphics[width=\textwidth]{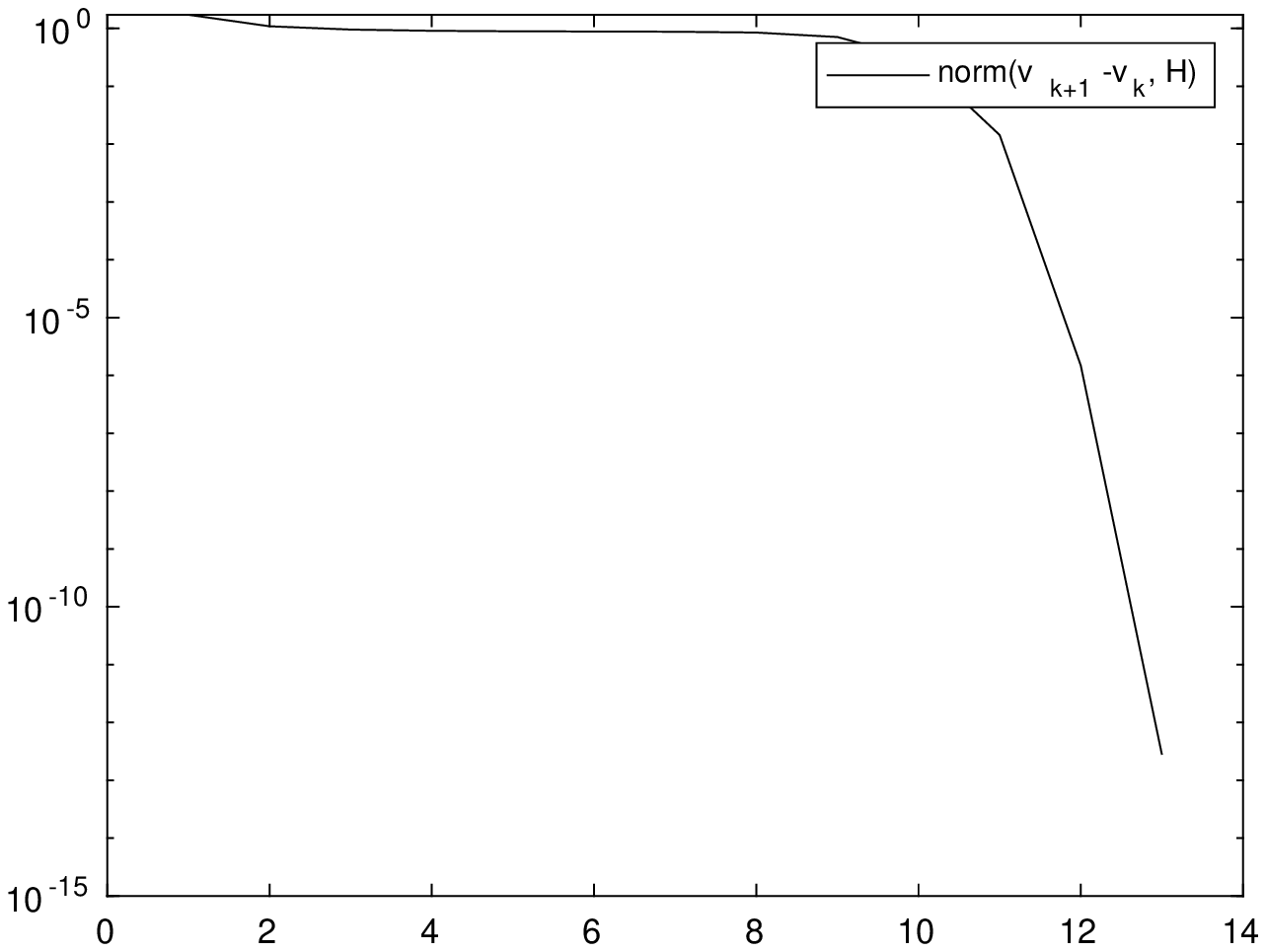}
    \caption{Figure 1b}
  \end{subfigure}
 \vspace{-0.5cm}  
   \end{figure} 
\end{example}
In principle  the matrix eigenvalue problem \eqref{mat_eig} can be solved by using
 any reliable method.   For large sparse matrices  in the case
 of clustered eigenvalues the convergence  of the inner iteration 
 can be greatly improved by  considering  variants  of the power method
 based on the Arnoldi/Krylov process  for approximating a few largest eigenvalues of the
 matrix.  In the next section the effectiveness and robustness of  these
 methods are evaluated by numerical experiments.

 \section{Experimental Setup}\label{sec3}
 We have tested the algorithms presented  above in a numerical environment using
 MatLab R2018b on a PC with 
Intel Core i7-4790 processor. The first method  to be considered is
 the  scaled SKK iteration implemented by  Algorithm \ref{alg1}. 
\begin{algorithm}[ht]
  \caption{Scaled SKK iteration}\label{alg1}
  \hspace*{\algorithmicindent} \textbf{Input:} $A\in \mathbb R^{n\times n}$, $A\geq 0$ and  a tolerance $\tau>0$ \\
 \hspace*{\algorithmicindent} \textbf{Output:} $\B x$ such that $\B x=T\B x$, $\B x \geq \B 0$, $\tt{sum}(\B x)=1$
\begin{algorithmic}[1]
\State \textbf{Function}  SKK$(A,\tau)$
\State $\B x=\tt{ones}(n,1)/n$;
\State $\textit{err}=inf$;
\While {$\textit{err}>\tau$}
\State $\B z=T\B x$;
\State $s=\tt{sum}(z)$;
\State $z=z/s$;
\State $\textit{err}=\tt{norm}(\B z-\B x)$;
\State $\B x=\B z$;
\EndWhile
\State \textbf{EndFunction}
\end{algorithmic}
\end{algorithm}

Based on the results of the previous section we propose to  exploit the properties of  either power method and
Arnoldi-type iterations  for computing the SKK vector.  The resulting schemes performs as follows: 
\begin{algorithm}[ht]
  \caption{Arnoldi-type  method}\label{alg2}
  \hspace*{\algorithmicindent} \textbf{Input:} $A\in \mathbb R^{n\times n}$, $A\geq 0$ and  given  tolerances $\tau>0$ \\
  \hspace*{\algorithmicindent} \textbf{Output:}  $\B x$ such that $\B x=T\B x$, $\B x \geq \B 0$, $\tt{sum}(\B x)=1$
\begin{algorithmic}[1]
\State \textbf{Function}  Arnoldi\_SKK$(A,\tau)$
\State $\B x=\tt{ones}(n,1)/n$;
\State $\textit{err}=inf$;
\While {$\textit{err}>\tau$}
\State $\left[\lambda,\B z\right]=\textit{FDE}(J_T(x), \B x, \tau)$;
\State $s=\tt{sum}(z)$;
\State $z=z/s$;
\State $\textit{err}=\tt{norm}(T\B z-\B z)$;
\State $\B x=\B z$; 
\EndWhile
\State \textbf{EndFunction}
\end{algorithmic}
\end{algorithm}

Algorithm \ref{alg2} makes use of an internal function $\textit{FDE}(J_T(x), \B x, \tau)$  for ``finding the dominant eigenpair''
of $J_T(x)$ at a prescribed tolerance depending on the value of $\tau$. If  $\textit{FDE}$ implements the power method then
Algorithm \ref{alg2} reduces to the SKK$_\infty$ iterative method.  However,
when the  largest eigenvalues of  $J_T(x)$ are clustered  the power method will perform poorly. In this case 
the performance  of  the eigensolver  can be improved by approximating  a few largest eigenvalues simultaneously  based
on Arnoldi methods.  In particular,
it  has been noted that  the orthogonalization of Arnoldi process achieves effective separation of eigenvectors \cite{YYN}.

The performances of the two algorithms have been evaluated and compared on sparse matrices.  It is worth pointing out that
MatLab  implements IEEE arithmetic and therefore, differently from the convention assumed at the beginning  of Section \ref{sec2},
 we find that $0\star(+Inf)=NaN$.  In some exceptional cases this discrepancy  can produce numerical difficulties and wrong results. 
Nevertheless, we have preferred  to avoid the redifinition of the multiplication operation  by  presenting tests
that are unaffected by such issue.
In our first set of  problems we  compute an approximation of the dominant eigenpair of $J_T(x)$ by using the
MatLab function {\tt eigs}  which implements an implicitly restarted Arnoldi method. The input sequence of {\tt eigs} is
given as

{{\scriptsize{
\begin{verbatim}
[V,D]=eigs(@(w)D2*AFUNT(D1*AFUN(w)),length(A),1,'largestabs','StartVector',x);
\end{verbatim}
}}}

where $AFUN(\B w)$  and $AFUNT(\B w)$   are functions that compute the product $A\B w$ and $A^T \B w$, respectively,
where  $A$ is stored in a sparse format. 
The test suite  consists of  the following matrices with entries 0 or 1 only:
\begin{enumerate}[i]
\item HB/can\_1072 of size $n=1072$ from the Harwell-Boeing collection;
\item SNAP/email-Eu-core of size $n=1005$ from the 
  SNAP (Stanford Network Analysis Platform) large network dataset collection;
  \item SNAP/Oregon-1 of  size $n=11492$  from the 
    SNAP  collection;
    \item SNAP/wiki-topcats matrix from the 
      SNAP  collection.  The original matrix has size $n=1791489$ but in order  to avoid paging issues  we
      consider here  its leading principal submatrix of order $n=32768$;
\end{enumerate}
The spy plots of these matrices are shown in Figure \ref{fig:sp}. 
\begin{figure}
  \begin{subfigure}[b]{0.4\textwidth}
    \includegraphics[width=\textwidth]{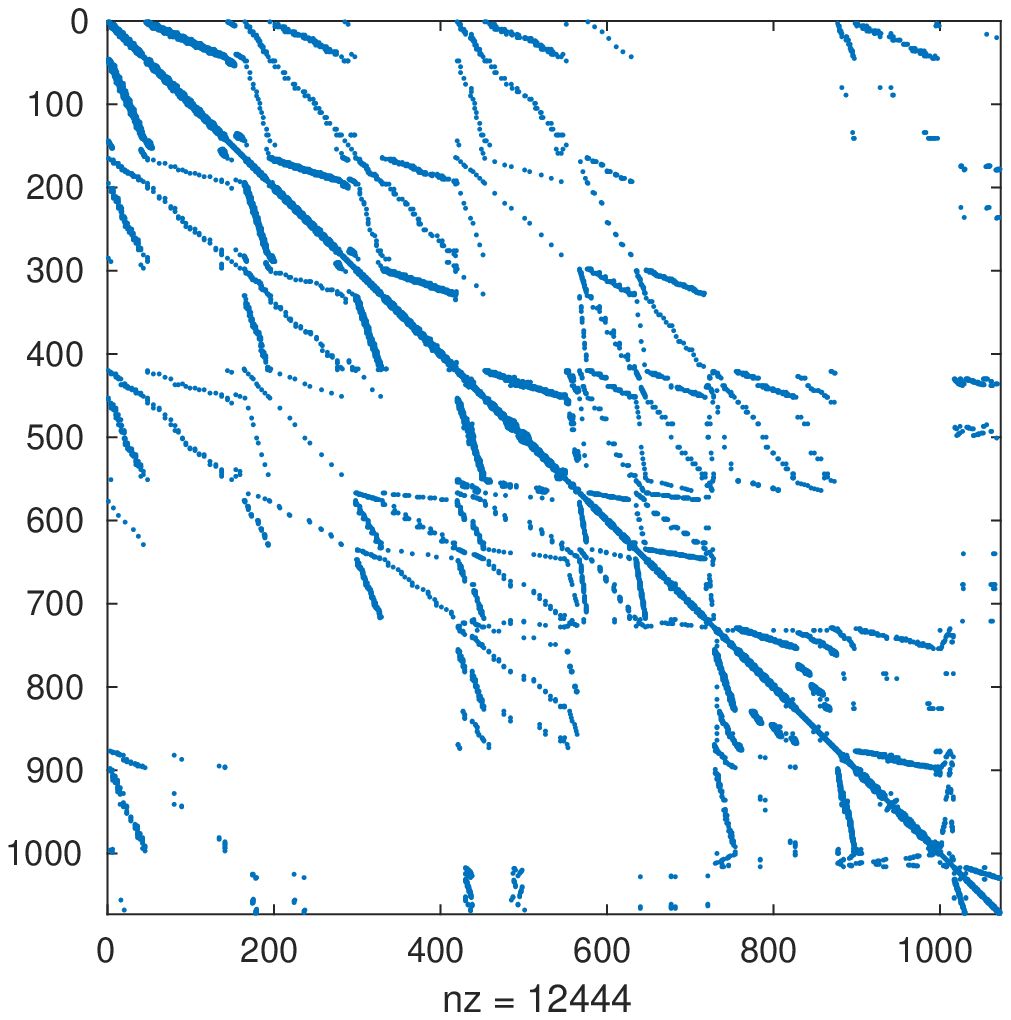}
    \caption{HB/can\_1072}
  \end{subfigure}
  \begin{subfigure}[b]{0.4\textwidth}
    \includegraphics[width=\textwidth]{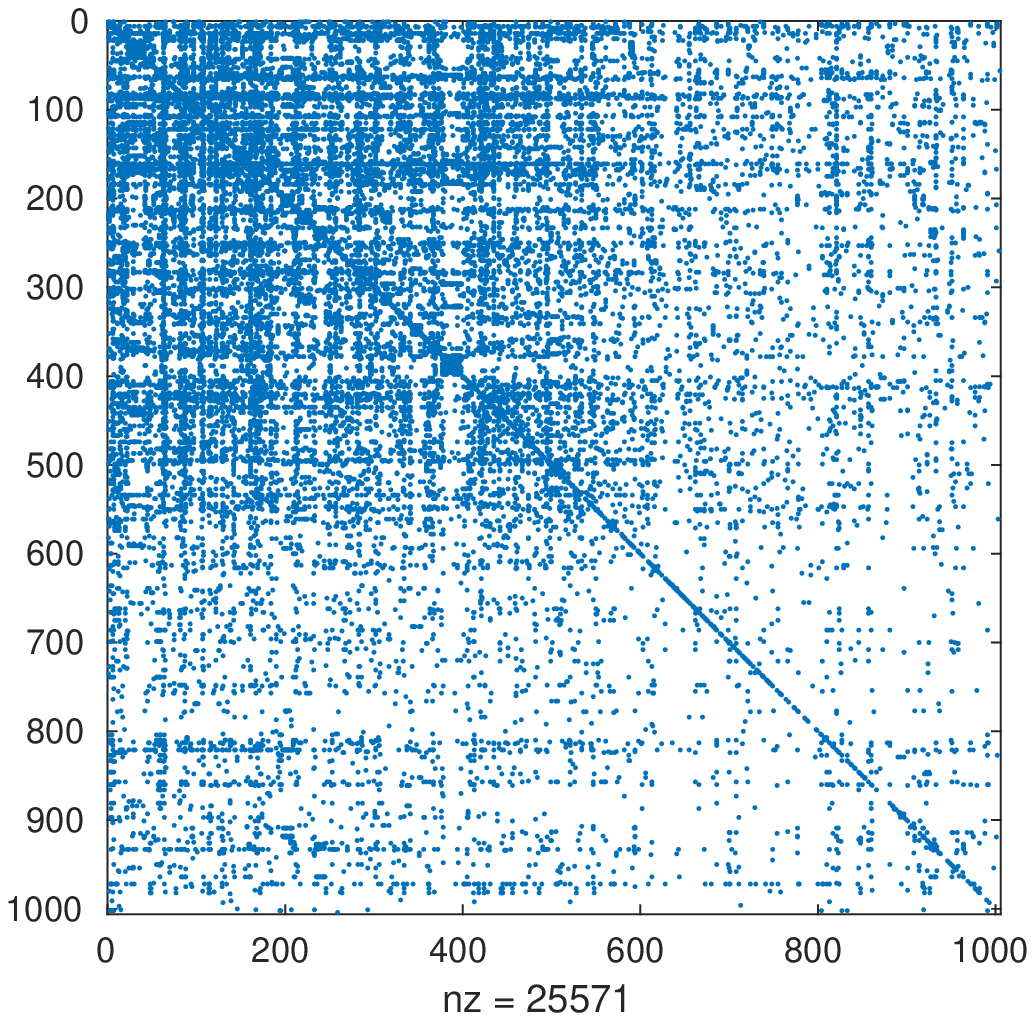}
    \caption{SNAP/email-Eu-core}
  \end{subfigure}
  \begin{subfigure}[b]{0.4\textwidth}
    \includegraphics[width=\textwidth]{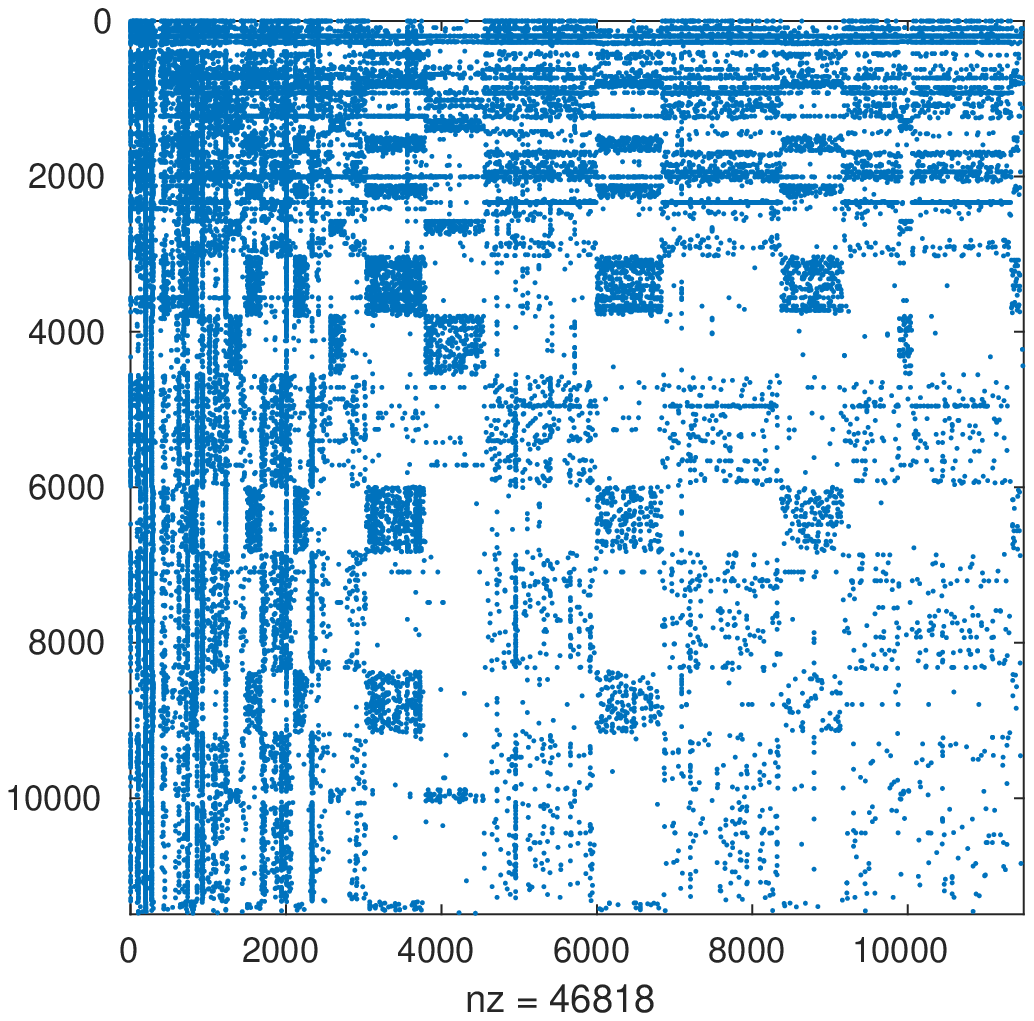}
    \caption{SNAP/Oregon-1}
  \end{subfigure}
  \hfill
  \begin{subfigure}[b]{0.4\textwidth}
    \includegraphics[width=\textwidth]{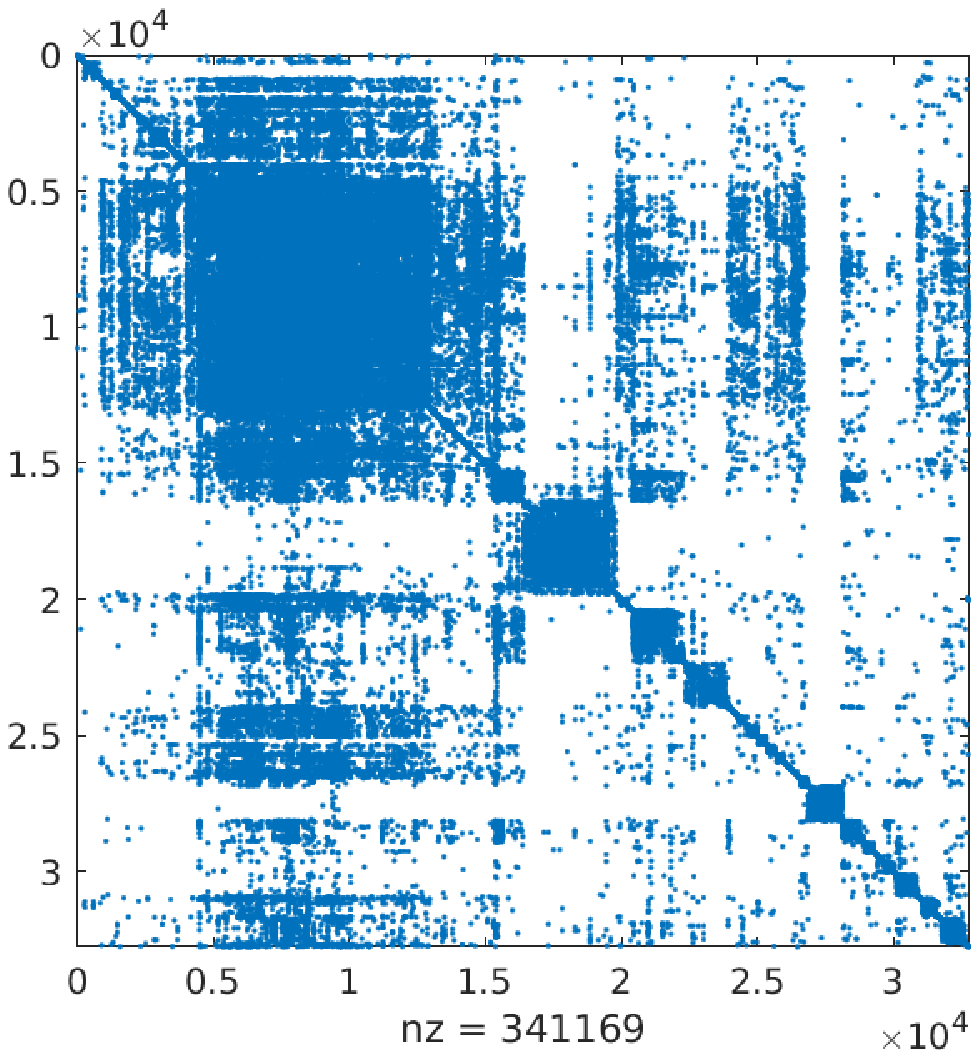}
    \caption{SNAP/wiki-topcats}
  \end{subfigure}
\caption{Spy Plots}  
\label{fig:sp}  
\end{figure}
The matrix HB/can\_1072 is sparse and irreducible.  The  (scaled) SKK iteration is convergent.  In Table \ref{t_ex_1}
we compare the computing times of Algorithm \ref{alg1} and \ref{alg2} for  different values of  $\tau$.
All  times are in seconds and  averaged  over 10 runs. 
  
\begin{table}[ht]
\caption{Computing times of Algorithm  \ref{alg1} and \ref{alg2} applied to HB/can\_1072 for different values of $\tau$}
\centering 
\begin{tabular}{|c| c| c| c|c|c|}
\hline
$\tau$ & 1.0e-6 & 1.0e-8& 1.0e-10& 1.0e-12& 1.0e-14\\ 
\hline
Alg1 &0.00296& 0.0147& 0.0315& 0.0484& 0.0657\\
\hline
Alg2 & 0.0151& 0.0213& 0.0213& 0.0253& 0.0254\\
\hline
\end{tabular}
\label{t_ex_1}
\end{table}  

As suggested at the beginning of the previous section the approach pursued by Algorithm \ref{alg2}  exhibits a
convergent behavior under the  same assumptions as the SKK iteration.  Moreover, for low levels of accuracy the cost of the
eigensolver is dominant whereas  Algorithm \ref{alg2} becomes  faster than Algorithm \ref{alg1} as $\tau$ decreases.

The remaining  matrices (ii), (iii) and  (iv)  from the
SNAP collections are more challenging due to the occurrence of clustered eigenvalues around 1  of
$J_T(\B x)$ where $\B x\in \mathcal P_0$ is a  fixed point of $T$.  According to \cite{KN}
in order to carry out  the approximation  of this vector efficiently we consider
perturbations of the  input matrix $A$
of the form
\[
\widetilde A=A + \gamma \B e\B e^T, \quad \B e=\left[1, \ldots, 1\right]^T, \quad \gamma>0, 
\]
for  decreasing values $\gamma_i$, $1\leq i\leq K$,  of $\gamma$. The approach resembles
the customary  strategy  employed for
solving the  PageRanking problem.   In the next tables \ref{t_ex_2},  \ref{t_ex_3} and
\ref{t_ex_4}  we   report  the computing times of Algorithm  \ref{alg1} and \ref{alg2}.
When   $\gamma=\gamma_1$   both algorithms start  with $\B x=\B e/n$  whereas  for $i>1$ the
starting vector is  given by the  solution  computed at  the previous step with $\gamma=\gamma_{i-1}$.
In all experiments the tolerance  was set at $\tau=1.0e-12$.

\begin{table}[ht]
\caption{Computing times of Algorithm  \ref{alg1} and \ref{alg2} applied to  SNAP/email-Eu-core for different values of $\gamma$}
\centering 
\begin{tabular}{|c| c| c| c|c|c|c|c|}
\hline
$\gamma$ & 1.0e-2 & 1.0e-4& 1.0e-6& 1.0e-8& 1.0e-10& 1.0e-12&1.0e-14\\ 
\hline
Alg1 &0.004& 0.009& 0.04& 0.34& 2.84 &23.24 & 185.54\\
\hline
Alg2 & 0.01& 0.02& 0.12& 0.18& 0.42 & 0.73& 1.16\\
\hline
\end{tabular}
\label{t_ex_2}
\end{table}

\begin{table}[ht]
\caption{Computing times of Algorithm  \ref{alg1} and \ref{alg2} applied to SNAP/Oregon-1  for different values of $\gamma$}
\centering 
\begin{tabular}{|c| c| c| c|c|c|c|c|}
\hline
$\gamma$ & 1.0e-2 & 1.0e-4& 1.0e-6& 1.0e-8& 1.0e-10 &1.0e-12& 1.0e-14\\ 
\hline
Alg1 &0.006& 0.008& 0.04& 0.32& 2.66 &21.86 & 174.27\\
\hline
Alg2 & 0.04& 0.06& 0.143& 0.7& 2.44 & 6.81& 12.28\\
\hline
\end{tabular}
\label{t_ex_3}
\end{table}

\begin{table}[ht]
\caption{Computing times of Algorithm  \ref{alg1} and \ref{alg2} applied to SNAP/wiki-topcats for different values of $\gamma$}
\centering 
\begin{tabular}{|c| c| c| c|c|c|c|c|}
\hline
$\gamma$ & 1.0e-2 & 1.0e-4& 1.0e-6& 1.0e-8& 1.0e-10 &1.0e-12&1.0e-14\\ 
\hline
Alg1 &0.02& 0.02& 0.18&0.69& 5.96&46.26&365.56 \\
\hline
Alg2 & 0.09& 0.2& 0.78&3.15& 11.59 &38.43 & 91.07 \\
\hline
\end{tabular}
\label{t_ex_4}
\end{table}

We observe that Algorithm \ref{alg2} outperforms Algorithm \ref{alg1}  for sufficiently small values of $\gamma$  when the perturbed matrix is close  to the original web link graph.

The second  set of test problems consists of matrices   which can be reduced
by permutation of rows and columns to block  triangular form.
The  reduction of the adjacency matrix of a graph in a block triangular form
is  related with the Dulmage-Mendelsohn decomposition \cite{DM}, which is a canonical decomposition
of a bipartite graph based on the notion of matching. The SKK iteration applied to block triangular matrices can not converge.
Depending on the number of  blocks  the scalar Arnoldi method employed  by the function {\tt eigs} can also performs poorly.
In this situation it can be recommended the use of a block  Arnoldi-based eigensolver which using  a set of starting vectors
is  able to compute multiple or clustered eigenvalues more efficiently than an unblocked routine.   In our experiments we
consider the  function {\tt ahbeigs} \cite{BJ}  which implements
a block Arnoldi method for computing a few eigenvalues of sparse matrices. Block methods can suffer from the occurrence of complex
eigenpairs.   Therefore, 
based on the proof of Theorem \ref{norm}
the method is applied to the symmetric matrix $G\cdot G^T$, $G=\diag(T\B x)\cdot A^T \cdot \diag(S\B  x)$
which is similar to $J_T(\B x)$.
The input sequence of {\tt ahbeigs} is
given as

{\scriptsize{
\begin{verbatim}
OPTS.sigma='LM';OPTS.k=m;OPTS.V0=R0;[V,D]=ahbeigs('afuncsym', n, speye(n),  OPTS)
\end{verbatim}
}}

where $n$ is the size of the matrix $A$, $m$ is the number of desired eigenvalues, $R0\in \mathbb R^{n\times m}$ is the set
of starting vectors and  {\tt 'afuncsym'}  denotes a function 
that computes the product  of $G\cdot G^T$ by a vector   where  $A$ is stored in a sparse format. 

For numerical testing we consider the following  matrices:
\begin{enumerate}
 \item the adjacency matrix  $A_{jazz}\in \mathbb R^{198\times 198}$
constructed from a
collaboration network between Jazz musicians.  Each node is a Jazz musician and an edge denotes that
two musicians have played together in a band.  The data was collected in 2003 \cite{jazz}.
The  MatLab command {\tt dmperm}  applied to the  jazz matrix    computes its
Dulmage-Mendelsohn decomposition.  It is found that the  permuted matrix  is
block triangular with 11 diagonal blocks;
\item  the matrix $A_{mbeause}\in \mathbb R^{496\times 496}$ generated by taking the absolute value of the
  matrix HB/mbeause  from the  the Harwell-Boeing collection. The original matrix  is  derived from
  an economic model which reveals several communities. This structure is  maintained in the modified matrix.
  The  MatLab command {\tt dmperm}  applied to  $A_{mbeause}$ returns a permuted matrix with 28 diagonal blocks. 
\end{enumerate}

In Figure \ref{fig2:sp}  we illustrate the spy plots of  the input matrices and their permuted versions.  
\begin{figure}
  \begin{subfigure}[b]{0.4\textwidth}
    \includegraphics[width=\textwidth]{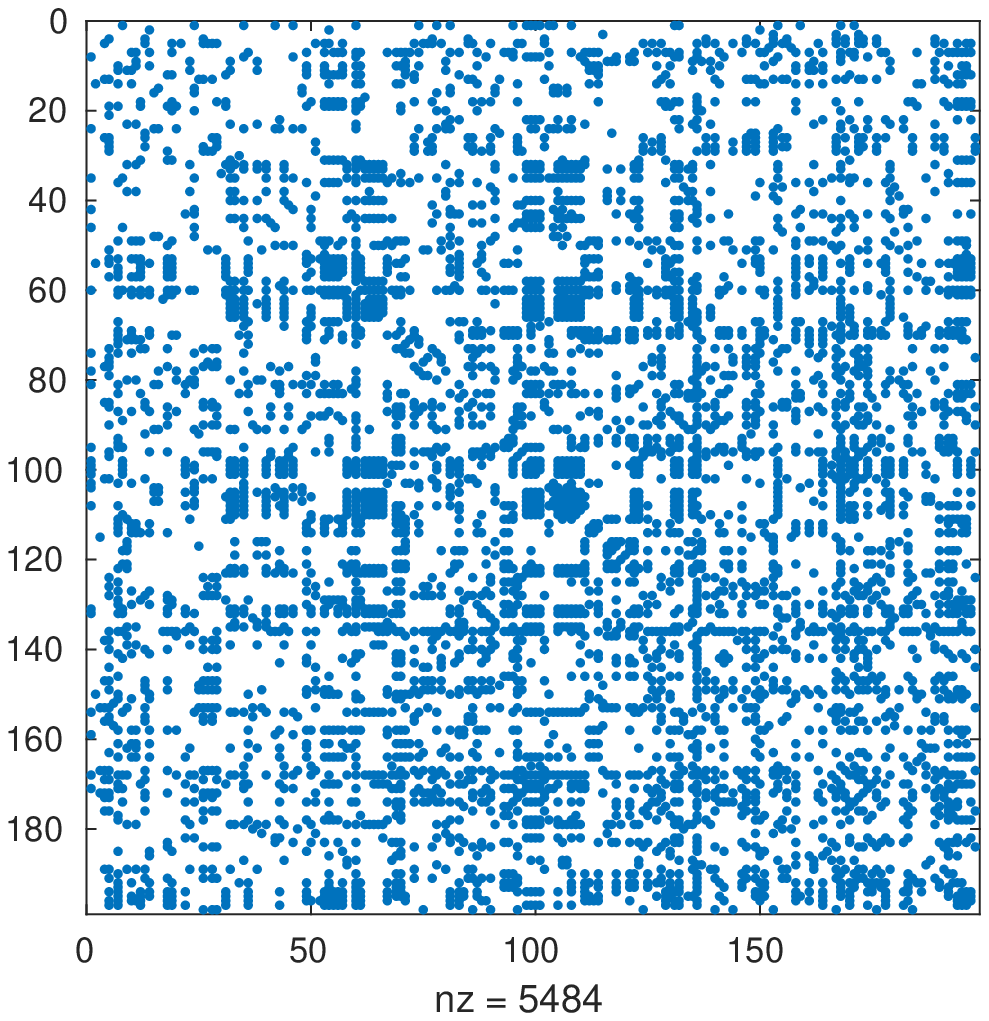}
    \caption{$A_{jazz}$}
  \end{subfigure}
  \begin{subfigure}[b]{0.4\textwidth}
    \includegraphics[width=\textwidth]{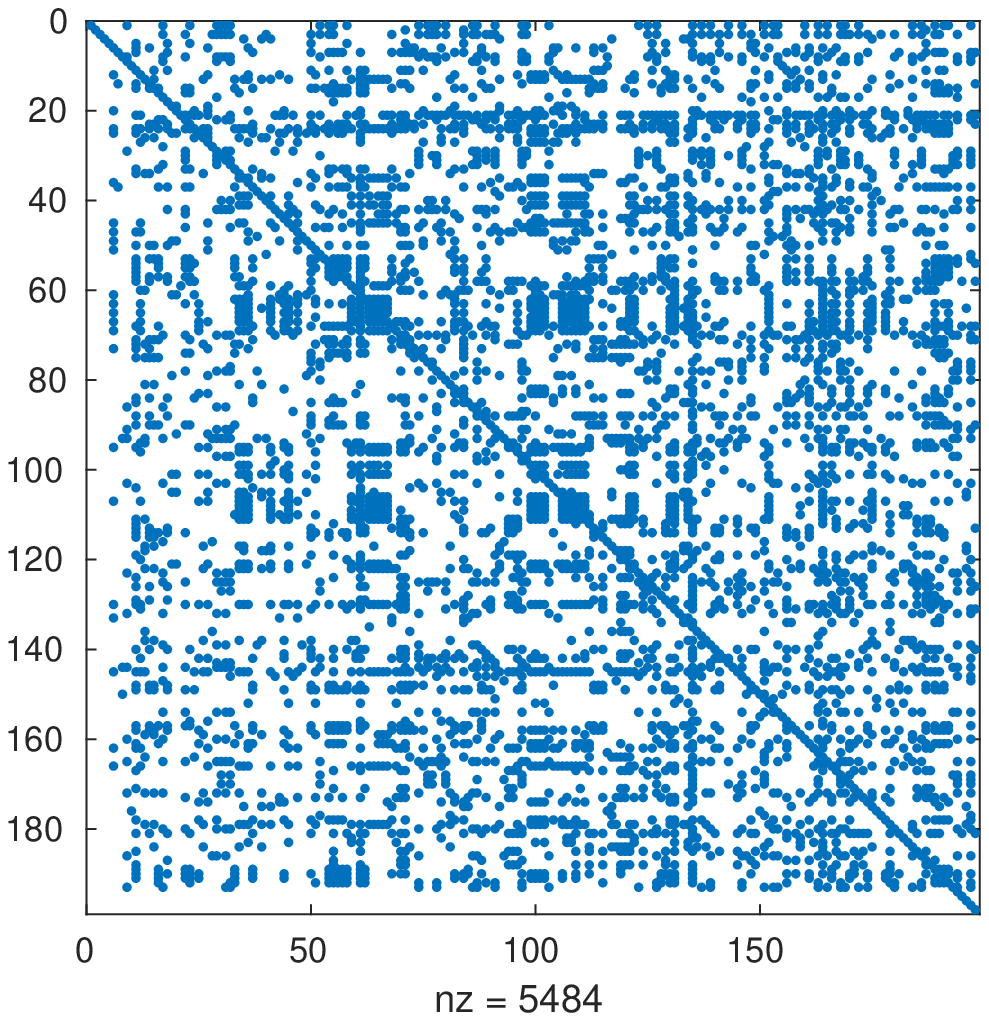}
    \caption{Permuted $A_{jazz}$}
  \end{subfigure}
  \begin{subfigure}[b]{0.4\textwidth}
    \includegraphics[width=\textwidth]{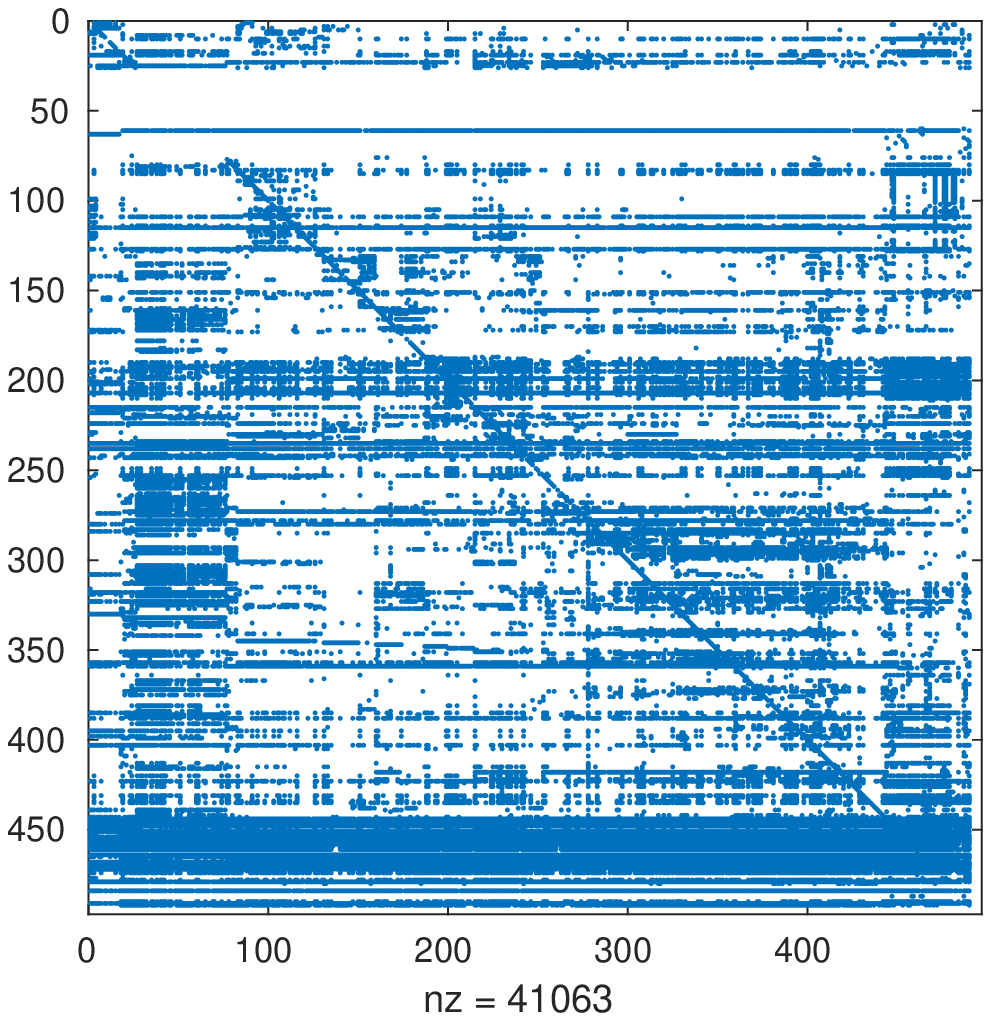}
    \caption{$A_{mbeause}$}
  \end{subfigure}
  \hfill
  \begin{subfigure}[b]{0.4\textwidth}
    \includegraphics[width=\textwidth]{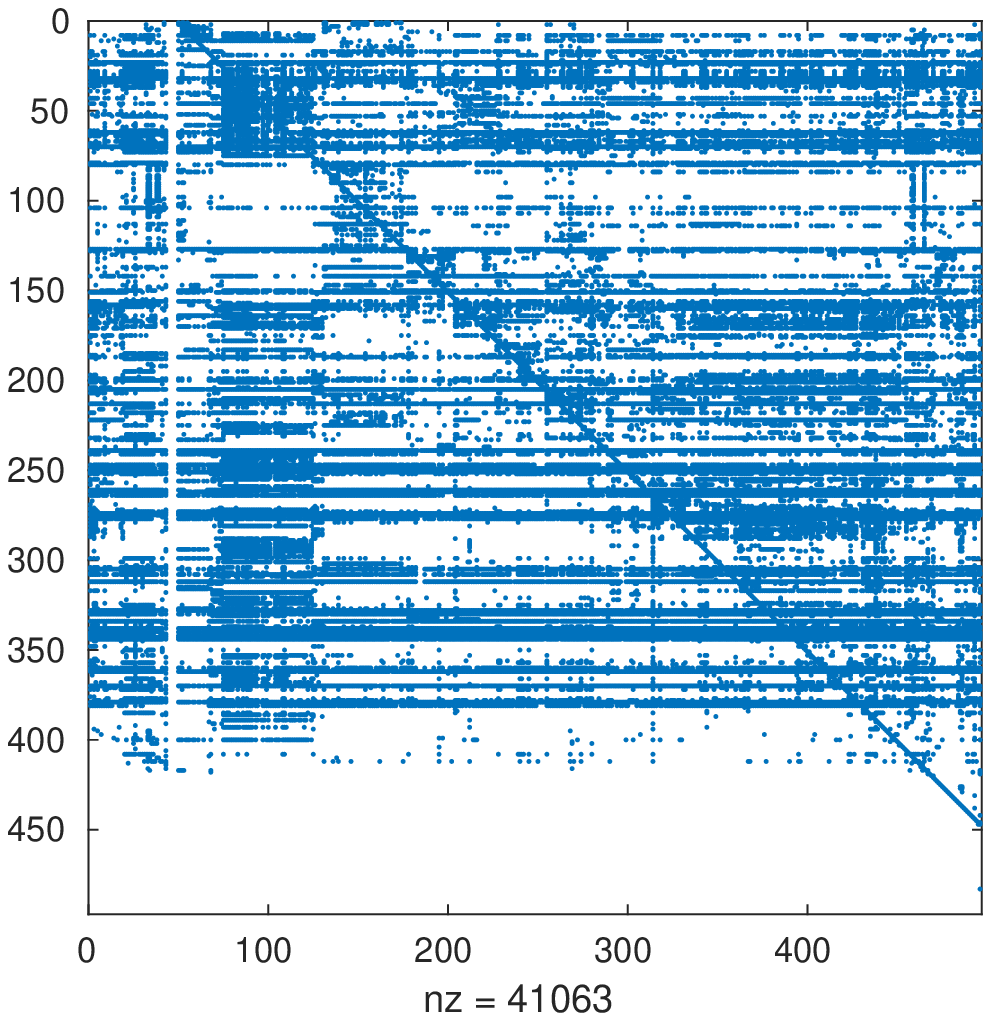}
    \caption{Permuted $A_{mbeause}$}
  \end{subfigure}
\caption{Spy Plots}  
\label{fig2:sp}  
\end{figure}

In Table \ref{t_ex_5} we compare the  computing times of   Algorithm  \ref{alg1} and \ref{alg2},  where
{\textit{FDE}} makes use of {\tt ahbeigs},  applied to the matrices $\widetilde A\colon=A+\gamma \B e\B e^T$ for
different values of $\gamma$ and $A=A_{jazz}, A_{mbeause}$.  In each experiment  the (block) starting vector is
$X={\tt rand}(n,m)$  where $n$ is the size of $A$, $m=1$ for Algorithm  \ref{alg1} and $m=16,32$ for  Algorithm  \ref{alg2}
applied to $A_{jazz}$ and $A_{mbeause}$, respectively.  

\begin{table}[ht]
\caption{Computing times of Algorithm  \ref{alg1} and \ref{alg2}   for different values of $\gamma$}
\centering 
\begin{tabular}{|c|| c| c| c|c||c|c|c|c|}
  \hline
   \multicolumn{1}{|c||}{} &
      \multicolumn{4}{|c||}{$A_{jazz}$} &
      \multicolumn{4}{|c|}{$A_{mbeause}$}\\
      \hline
$\gamma$ & 1.0e-8 & 1.0e-10& 1.0e-12& 1.0e-14& 1.0e-8 &1.0e-10& 1.0e-12& 1.0e-14\\ 
\hline
Alg1 &0.06& 0.18& 0.68& 2.87& 1.87 &14.43 & 112.26 & 808.26\\
\hline
Alg2 & 0.22& 0.24& 0.26& 0.36& 1.84 & 2.09&2.73 &3.61\\
\hline
\end{tabular}
\label{t_ex_5}
\end{table}

For these matrices  the methods
based on eigenvalue computations can be  dramatically faster than the fixed point iteration. 
  In particular,  the modified Algorithm \ref{alg2} applied to  $A_{jazz}$ with $\gamma=\tau=1.0-e-12$ and $m=16$ 
 converges in 14 iterations. In Figure \ref{f3} we show the error behavior as well  as  the singular values of the
balanced matrix (compare with  Theorem \ref{rate}).
\begin{figure}
    \begin{subfigure}{0.3\textwidth}
\includegraphics[scale=0.45]{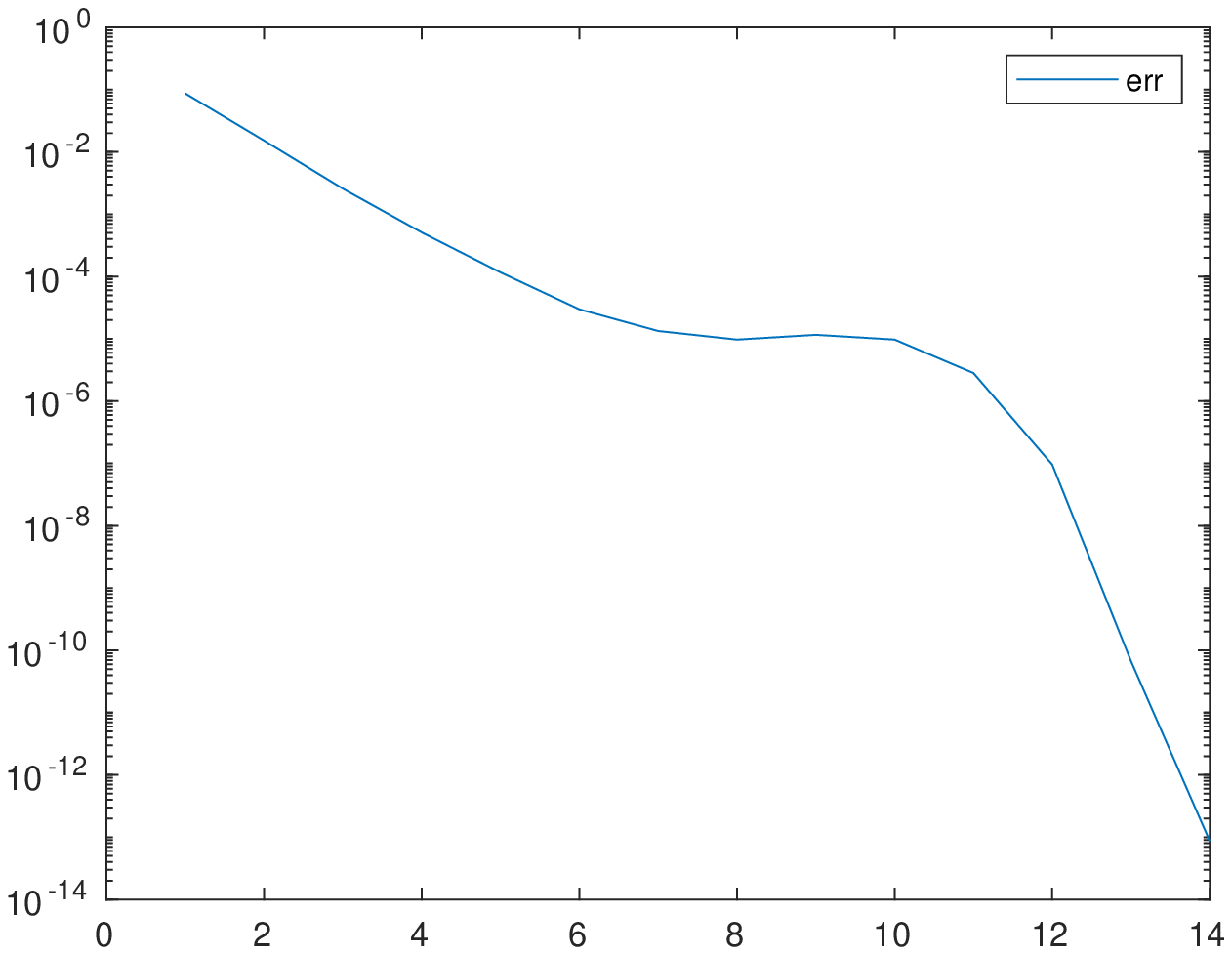}
    \end{subfigure}\hspace{0.2\textwidth}
    \begin{subfigure}{0.3\textwidth}
\includegraphics[scale=0.45]{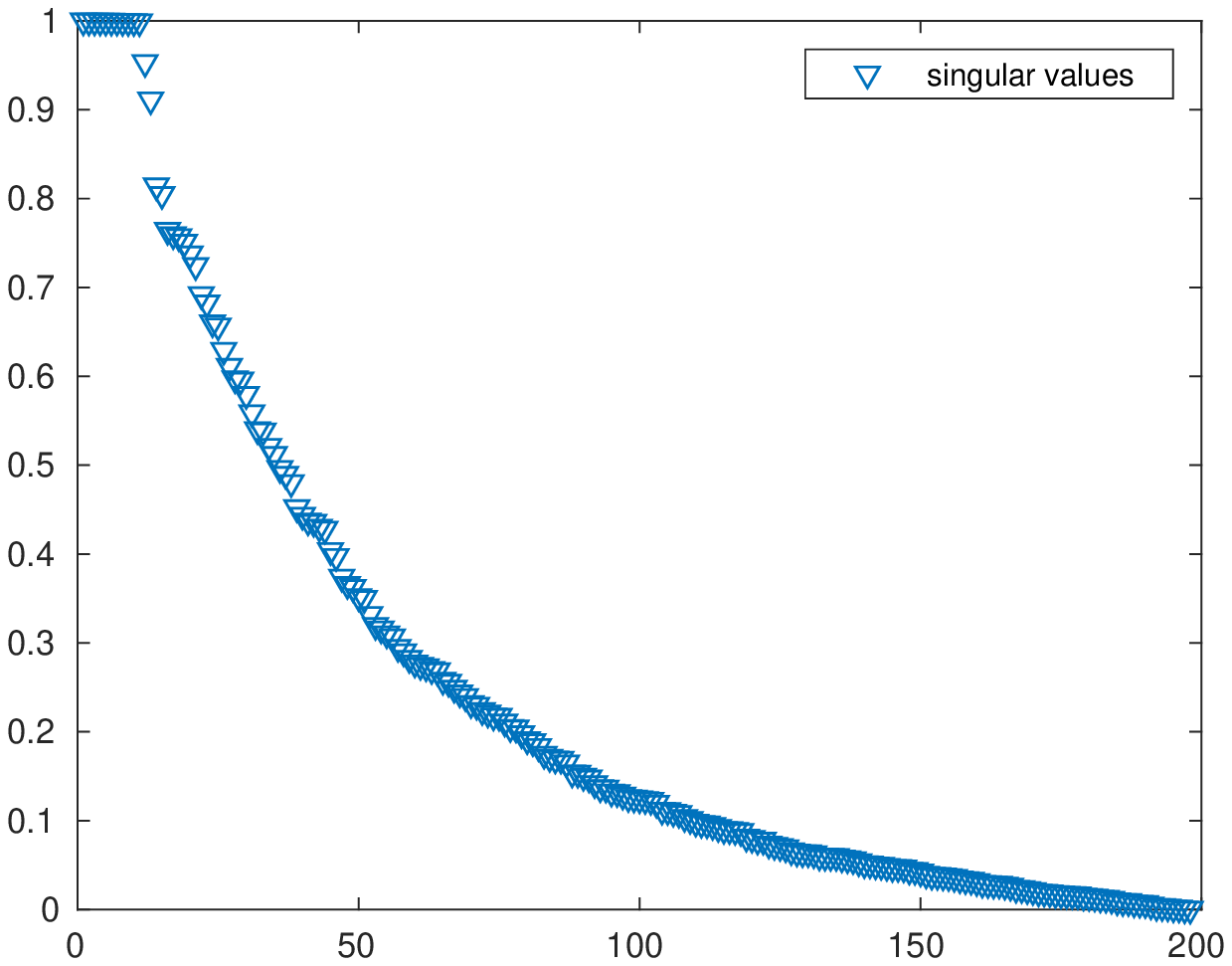}
\end{subfigure}
\caption{Error plot of modified Algorithm \ref{alg2} applied to the perturbed
jazz matrix and singular values of the  resulting balanced matrix}
\label{f3}
\end{figure}
It is worth stressing the accordance of the number of blocks  in the permuted matrix  with  the size of the
 cluster of singular values of the  balanced permuted matrix.

\section{Conclusions and  Future Work}\label{sec4}
In this paper we have discussed some numerical techniques for accelerating the customary SKK  iteration based on
certain equivalent formulations of the fixed point problem as a matrix eigenvalue problem.
Variants of the power  method relying upon the Arnoldi process
have been proposed for the efficient solution of the matrix eigenvalue problem.
There are several topics which remain to be addressed. Specifically: 
\begin{enumerate}
\item  A formal proof  of  the convergence for the SKK$_\infty$ method is still missing. As suggested  by
  Figure \label{fig:2} in this respect it might  be useful to
  investigate the properties of  the map
  $E\colon \mathcal P_0 \rightarrow \mathcal P_0$ defined  by $E(\B v) =\B w$ where
  $\B w$ is the normalized  dominant eigenvector of $J_T(\B v)$.  
  \item Theoretical results would  be extended to nonnegative matrices under customary assumptions on
    their supports.   Such extension can be based on the perturbative analysis  introduced in \cite{Menon1}
    and Section 6.2 of  \cite{LN_book}.
\item  The efficiency of the balancing schemes depends  on eigenvalue (singular value) clustering properties.
  Numerical experiments have revealed  a close connection between the 
  clustering of singular values of the balanced adjacency matrix and the clustering of nodes (community detection)
 in the corresponding graph. Relations  with the  block triangular form (BTF) form of adjacency matrices  have also appeared.
  The possible use of   balancing schemes for detecting
  the  block structure of  adjacency matrices is an ongoing research work. 
\item  Finally, we plan to  study the numerical behavior of  block Arnoldi based methods
  by performing extensive numerical experiments with large  sparse  and data-sparse matrices. In  particular
  following \cite{GoGr} we can take advantage of knowing the largest eigenvalue of the limit problem to
  speed up the intermediate steps.
  \end{enumerate}


\end{document}